\documentclass[a4paper,11pt,reqno]{amsart}
\usepackage{amssymb}
\usepackage{enumitem}
\usepackage[alphabetic]{amsrefs}
\usepackage{tikz}
\usepackage{blindtext, xcolor}
\usepackage{comment}
\usepackage{mathrsfs}
\usepackage{geometry}
\usepackage{a4wide}
\usepackage{fullpage}
\usepackage{mathtools}

\usepackage[colorlinks]{hyperref}

\makeatletter
\@namedef{subjclassname@2020}{%
  \textup{2020} Mathematics Subject Classification: 5C31 primary; 82B20, 68W25 secondary} 
\makeatother

\newtheorem{corollary}{Corollary}[section]
\newtheorem{theorem}[corollary]{Theorem}
\newtheorem{lemma}[corollary]{Lemma}
\newtheorem{proposition}[corollary]{Proposition}

\newtheorem{remark}[corollary]{Remark}

\newtheorem*{theorem*}{Theorem}
\newtheorem*{corollary*}{Corollary}

\numberwithin{equation}{section}

\usepackage{graphicx} 

\title{Improved bounds on the zeros of the chromatic polynomial of graphs and claw-free graphs}
 \author[F. Bencs]{Ferenc Bencs}
 \address{Ferenc Bencs, Centrum Wiskunde \& Informatica, P.O. Box 94079 1090 GB Amsterdam, The Netherlands}
 \email{\texttt{ferenc.bencs@gmail.com}}
 \thanks{FB was funded by the Netherlands Organisation of Scientific Research (NWO): VI.Veni.222.303}
 \author[G.Regts]{Guus Regts}
 \address{Guus Regts, Korteweg de Vries Institute for Mathematics, University of Amsterdam. P.O. Box 94248  
 1090 GE Amsterdam The Netherlands}
 \email{\texttt{guusregts@gmail.com}}
 \thanks{GR was funded by the Netherlands Organisation of Scientific Research (NWO): VI.Vidi.193.068}
\date{\today}

\begin{document}

\begin{abstract}
We prove that for any graph $G$ the (complex) zeros of its chromatic polynomial, $\chi_G(x)$, lie inside the disk centered at $0$ of radius $4.25 \Delta(G)$, where $\Delta(G)$ denotes the maximum degree of $G$.
This improves on a recent result of Jenssen, Patel and Regts, who proved a bound of $5.94\Delta(G)$.
Moreover, we show that for graphs of sufficiently large girth we can replace $4.25$ by $3.60$ and for claw-free graphs we can replace $4.25$ by $3.81$.

Our proofs add some substantially novel ideas to those developed by Jenssen, Patel, and Regts, while building on them.
A key novel ingredient for claw-free graphs is to use a representation of the coefficients of the chromatic polynomial in terms of the number of certain partial acyclic orientations.
\end{abstract}
\maketitle

\section{Introduction}
In this paper, we present improved bounds on the zeros of the chromatic polynomial of graphs of bounded maximum degree.
We recall that the chromatic polynomial of an $n$-vertex graph $G$, $\chi_G(x)$, is the unique monic polynomial of degree $n$ such that for any positive integer $q$, $\chi_G(q)$ is equal to the number of proper $q$-colorings of $G$.
Existence and uniqueness follow from the \emph{deletion-contraction identity} satisfied by the number of proper $q$-colorings: for any edge $e$ of $G$ we have
\begin{equation}\label{eq:classic del con}
    \chi_G(q)=\chi_{G\setminus e}(q)-\chi_{G/e}(q),
\end{equation}
where $G/e$ resp. $G\setminus e$ denote the graphs obtained from $G$ by contracting resp. deleting the edge $e$.
The identity~\eqref{eq:classic del con} consequently holds for any $q\in \mathbb{C}$.

Ever since the introduction of the chromatic polynomial in 1912 by Birkhoff~\cite{Birk}, the (location of the) zeros of the chromatic polynomial -- henceforth called \emph{chromatic zeros} -- have been intensively studied.
This started with unsuccessful attempts of Birkhoff~\cite{Birk}, to prove  what was then the four color conjecture.

A motivation for the study of the location of complex chromatic zeros originated in statistical physics from the Lee-Yang~\cite{yang1952statistical} and Fisher~\cite{fisher} approach to phase transitions, because the chromatic polynomial can be viewed as the zero temperature limit of the partition function of the anti-ferromagnetic Potts model.
This has resulted in several results for the chromatic zeros on lattice like graphs and other families of recursively defined graphs, see for example~\cites{Shrocketal1998,Shrock2001,Sokaldense,Sokaletal2001,Sokaletal2003,Sokaletal2009,ChioRoeder,BHRseriesparallelroots23} and references in there.

As mentioned above, the focus in the present paper is on chromatic zeros of the family of graphs of bounded maximum degree.
About 25 years ago,  Sokal~\cite{sokal} proved that there exists a constant $K\le7.97$ such that any zero of the chromatic polynomial of a graph $G$ is contained in the disk of radius $K\cdot \Delta(G)$ in the complex plane centered at $0$. This confirmed a conjecture of Biggs, Darmerell and Sands~\cite{BDS72} for regular graphs and Brenti, Royle and Wagner~\cite{BRW94} for general graphs.
Sokal's constant was subsequently improved to $K\leq 6.91$ by Fern\'andez and Procacci~\cite{Fernandez} and Jackson, Procacci and Sokal~\cites{JacksonSokal}. 
Very recently, the constant was further improved to $K\leq 5.94$ by Jenssen, Patel and Regts~\cite{JPR24}. 
Moreover, they showed that for graphs of girth at least $g$ there exists a constant $K_g$ such that $\lim_{g\to \infty} K_g\leq 3.86$ and such that the complex zeros of the chromatic polynomial of graphs $G$ of girth at least $g$ are contained in a disk of radius $K_g\cdot \Delta(G)$.
These results were subsequently refined by Fialho, Juliano and Procacci~\cite{procacci2024zero}, providing an increasing sequence (in terms of $\Delta)$ of constants $K_{g,\Delta}<K_g$ satisfying $\lim_{\Delta\to \infty} K_{g,\Delta}= K_g$ such that all complex zeros of the chromatic polynomials of graphs $G$ of girth at least $g$ and maximum degree at most $\Delta$ are contained in the disk of radius $K_{g,\Delta}\cdot \Delta$.

We should note that the constant $K$ has to be at least $1$, as is witnessed by the complete graph on $\Delta+1$ vertices (its chromatic polynomial has a root at $\Delta$). It must in fact be at least $1.59$, since complete bipartite graphs $K_{\Delta,\Delta}$ have chromatic roots of modulus at least $1.59 \Delta$ (for large values of $\Delta$), see~\cite{RoyleSokalmaxmax}*{Footnote 4}.
Royle~\cite{Roylesurvey}*{Conjecture 6.6} actually conjectured  that complete bipartite graphs are the extremal graphs for this problem. 

This brings us to the results of the present paper.
\subsection*{Our results}
Our main result is a further improvement on the constant $K$:
\begin{theorem}\label{thm:main}
Let $G$ be a graph. If $\chi_G(z)=0$ for $z\in \mathbb{C}$, then $|z|\leq 4.25 \Delta(G).$
\end{theorem}
For small values of $\Delta$ we can actually do better than $4.25$.
Define for any integer $\Delta\geq 3$, $K_\Delta$ to be the smallest number such that for all graphs $G$ of maximum degree at most $\Delta$ all roots of $\chi_G(q)$ are contained in the disk of radius $K_{\Delta}\cdot \Delta$.
In Table~\ref{tab:main general} we present upper bounds on $K_\Delta$ for small values of $\Delta$, significantly improving on the results from~\cite{procacci2024zero}. 

We note that Theorem~\ref{thm:main} and the bounds in Table~\ref{tab:main general} also improve on the best known bounds for the \emph{real} zeros of the chromatic polynomial of bounded degree graphs~\cite{Dongrealzeros}.

\begin{table}[h]
    \centering
    \begin{tabular}{c|c|c|c}
        $\Delta$ & $K_\Delta$ & a & b \\
        \hline
        3  &  2.321  &  0.388  &  1.207 \\
        4  &  2.816  &  0.367  &  0.990 \\
        5  &  3.107  &  0.358  &  0.913 \\
        6  &  3.298  &  0.353  &  0.874 \\
        20  &  3.965  &  0.340  &  0.779 \\
        100  &  4.192  &  0.334  &  0.745
    \end{tabular}
    \caption{Upper bounds on $K_\Delta$ obtained for small values of $\Delta$. Here $a$ and $b$ certify the value of $K_\Delta$ by Theorem~\ref{thm:main general}.}
    \label{tab:main general}
\end{table}
\subsubsection*{Large girth graphs}
Similarly as in the works~\cites{JPR24,procacci2024zero}, we can incorporate the dependence on the girth into the constant $K$ (resp. $K_\Delta$). Although we won't pursue the best values for finite girth, rather we only compute bounds on the constant for large girth.
For positive integers $\Delta,g\geq 3$ we denote by $K_{\Delta,g}$ the smallest number such that for all graphs $G$ of maximum degree at most $\Delta$ and girth at least $g$ all roots of $\chi_{G}(q)$ are contained in the disk of radius $K_{\Delta,g}\cdot \Delta$.
Clearly, $\{K_{\Delta,g}\}_{g\geq 3}$ is a decreasing sequence. We let $K_{\Delta,\infty}:=\lim_{g\to \infty}K_{\Delta,g}$. 

\begin{theorem}\label{thm:main large girth}
    Let $\Delta\ge 3$ be an integer. Then $K_{\Delta,\infty}\leq \tfrac{1}{W(e^{-1})}\approx 3.59$, where $W(z)$ is the Lambert W-function.
\end{theorem}
Again, for small values of $\Delta$ we can improve on this bound.
Below in Table~\ref{tab:main large girth}, we have collected upper bounds on $K_{\Delta,\infty}$ for several values of $
\Delta$.
\begin{table}[h]
    \centering
    \begin{tabular}{c|c|c|c}
        $\Delta$ & $K_{\Delta,\infty}$ & a & b \\
        \hline
        3  &  1.944  &  0.314  &  2 \\
        4  &  2.364  &  0.286  &  3/2\\
        5  &  2.612  &  0.274  &  4/3\\
        6  &  2.776  &  0.267  &  5/4 \\
        20  &  3.348  &  0.249  &  19/18 \\
        100  &  3.547  &  0.245  &  99/98
    \end{tabular}
    \caption{Upper bounds on $K_{\Delta,\infty}$ obtained for small values of $\Delta$. Here $a$ and $b$ certify the value of $K_{\Delta,\infty}$ by Proposition~\ref{prop:main general large girth}.}
    \label{tab:main large girth}
\end{table}

\subsubsection*{Claw-free graphs}
Recall that a graph is called \emph{claw-free} if it does not contain the tree on $4$ vertices with $3$ leaves as an induced subgraph.
Apart from being a natural family of graphs extending line graphs, we also have a more technical motivation for looking at claw-free graphs, which has to do with our proof of the theorem below. We will say a few words about this in the final section.

We show that for claw-free graphs our bound on $K$ can be improved even further:

\begin{theorem}\label{thm:main CF}
Let $G$ be a claw-free graph. If $\chi_G(z)=0$ for $z\in \mathbb{C}$, then $|z|\leq 3.81 \Delta(G)$.
\end{theorem}

We next discuss a motivation from theoretical computer science for studying absence of chromatic zeros, after which we say a few words about our approach to proving these results and provide an organization of the remainder of the paper.

\subsection*{Motivation from computer science.}
Recently, there has been a motivation from theoretical computer science in understanding the location of the complex chromatic roots of bounded degree graphs. 
The reason for that is that for families of bounded degree graphs, one can design efficient algorithms for approximating evaluations of the chromatic polynomial on unbounded open sets that do not contain chromatic zeros via Barvinok's interpolation method~\cites{barvinok, PR17}.
In particular, together with the interpolation method, our results imply that for each $\Delta\in \mathbb{N}$ and $z\in \mathbb{C}$ such that $|z|>4.25\Delta$ (resp. $|z|>3.81\Delta$), there exists an algorithm that on input of an $n$-vertex graph (resp. claw-free graph) of maximum degree at most $\Delta$ and $\varepsilon>0$ computes a relative $\varepsilon$-approximation to $\chi_G(z)$ in time polynomial in $n/\varepsilon$.
Here, a relative $\varepsilon$-approximation to $\chi_G(z)$ is a number $\xi$ such that $\chi_G(z)=e^{t}\xi$ for some $t\in \mathbb{C}$ such that $|t|\leq \varepsilon$.

For the design of efficient approximation algorithms for merely counting the number of proper $q$-colorings of graphs of maximum degree at most $\Delta$ (i.e., computing $\chi_G(q))$, the best known bounds on $q$ are of the form $q\geq 1.809\Delta$~\cites{carlson2025flip,chen2024deterministic}.
It is an important folklore conjecture that such algorithms exist for $q\geq \Delta+1$.
See also~\cite{FriezeVigodasurvey}.
From this perspective a more than twenty year old conjecture of Sokal (see \cite{Jackson}*{Conjecture 21}, and Sokal~\cite{Sokalsurvey}*{Conjecture 9.5$^{\prime \prime}$}) is very relevant.
The conjecture says that for a graph $G$, $\chi_G(x) \not= 0$ if ${\rm Re}(x) > \Delta(G)$. 
A positive resolution of this conjecture would, via the interpolation method, fully resolve this folklore conjecture.
So far, Sokal's conjecture is wide open.

\subsection*{Approach}
Our approach builds on the approach taken by Jenssen, Patel, and Regts~\cite{JPR24}, which in turn takes inspiration from the polymer method originating in statistical physics.
In~\cite{JPR24}, an interpretation of the coefficients of the chromatic polynomial in terms of the number of broken-circuit-free forests with a given number of edges was used.
Inspired by the polymer method they gave an inductive proof of zero-freeness by essentially looking at the ratio of the chromatic polynomial of a graph $G$ and that of $G-v$ for some vertex $v$. 
They expand this ratio as a summation over broken-circuit-free trees rooted at $v$ of products of inverse ratios for smaller graphs, which by induction can be appropriately bounded. 
Using bounds on rooted tree-generating function, the original ratio can be bounded.
Making use of the combinatorial properties of these broken-circuit-free forests, they managed to cleverly replace certain terms in this expansion, eventually leading to a bound that improved the bound of Fern\'andez and Procacci~\cite{Fernandez}.

Our proof of Theorem~\ref{thm:main} goes along similar lines. 
However, a crucial difference is that we look at the ratio of the polynomial of a graph $G$ and that of $G\setminus e$ for an edge $e$. 
Although at first this may seem like a minor difference, we would like to stress here that a naive implementation of this idea would result in bounds that are worse than by using the vertex based approach.
To actually obtain better bounds required new ideas; see Lemma~\ref{lem:expand G/e} below. 
Additionally, our edge based approach leads to a somewhat more involved analysis.

To prove our result for claw-free graphs we use a different interpretation of the coefficients of the chromatic polynomial, in terms of certain partial acyclic orientations due to Greene and Zaslavsky~\cite{greene1983interpretation}. 
We reinterpret these orientations in terms of certain forests (different from the broken-circuit-free forests) and proceed in similar way as in~\cite{JPR24} using induction on the number of vertices. 
Making use of basic properties of claw-free graphs and the family of forests this allows us to improve the bound on their chromatic zeros.

\subsection*{Organization}
The next section contains the two combinatorial interpretations of the coefficients of the chromatic polynomial that we use to prove our main results. 
Section~\ref{sec:trees} contains results on the tree-generating functions that we need. 
In Section~\ref{sec:general} we prove Theorems~\ref{thm:main} and~\ref{thm:main large girth} and in Section~\ref{sec:claw-free} we prove Theorem~\ref{thm:main CF}. We conclude with some remarks in Section~\ref{sec:remarks}.

\section{Interpretations of the coefficients of the chromatic polynomial} 
In this section we recall two different interpretations of the coefficients of the chromatic polynomial that will be used to prove our main theorems. 
We note that the second interpretation is only used to prove our result on claw-free graphs and can hence be skipped if one is only interested in our proof of Theorems~\ref{thm:main} and~\ref{thm:main large girth}.
\subsection{Whitney's theorem}
Whitney~\cite{Whitney} gives a combinatorial description of the coefficients of the chromatic polynomial in terms of so-called \emph{broken-circuit-free sets}. 
To define this notion we need a fixed ordering of the collection of edges of graph $G$.
A \emph{broken-circuit-free set} in $G$ is a forest $F \subseteq E$ such that for each edge $e$ such that $F\cup \{e\}$ contains a cycle, $e$ is \emph{not} the largest edge in that cycle.
From now on, we abbreviate ``broken-circuit-free'' by BCF and we write $\mathcal{F}_G$ for the set of all BCF sets in $G$ (including the empty set). 
We define the polynomial
 \begin{align}\label{eq:FGDef}
     F_G(x) = \sum_{F \in \mathcal{F}_G}x^{|F|}.
 \end{align}

Whitney proved that up to a simple transformation the polynomial $F_G$ is equal to the chromatic polynomial.
\begin{theorem}[Whitney \cite{Whitney}]\label{thm:whitney}
If $G$ is an $n$-vertex graph, then 
\[
\chi_G(x) = x^nF_G(-1/x). 
\]
\end{theorem}
Note that the theorem implies that even though $\mathcal{F}_G$ depends on the ordering of the edges of $G$, the polynomial $F_G$ does not.
We refer to~\cite{JPR24} for a short proof of Whitney's theorem based on the deletion-contraction recurrence.
For future reference we explicitly state the deletion-contraction recurrence for $F_G$ for a graph $G=(V,E)$. For any edge $e$ of $G$ that is not a loop we have
\begin{equation}\label{eq:deletion/contraction F}
    F_G(x)=F_G(x)+xF_{G/e}(x).
\end{equation}

\subsection{Acyclic orientations}\label{ssec:acyclic}
Another well known interpretations of the coefficients of $\chi_G$ is in term of \emph{acyclic orientations with a unique source} due to Greene and Zaslavsky~\cite{greene1983interpretation}*{Theorem 7.3} that we will now describe.
Let $H=(U,E)$ be a graph and let $v\in U$. 
An orientation $\omega$ of the edges of $H$ is called \emph{acyclic with unique source $v$} if the directed graph obtained from $H$ and $\omega$ does not contain directed cycles and has $v$ as its only source (i.e., all edges incident to $v$ are oriented away from $v$ while any other vertex has at least one edge oriented towards it).
We denote by $\text{AO}(H,v)$ the collection of all acyclic orientations of $H$ with $v$ as its unique source, and $\text{ao}(H,v)$ denotes its size. We note that $\text{ao}(H,v)=0$ in case $H$ is not connected (as otherwise there would be more than one source).

For a finite set $V$ we denote by $\mathcal{P}(V)$ the collection of partitions of $V$ into non-empty sets.
Let now $G=(V,E)$ be a graph equipped with an ordering of its vertices.
Consider the following polynomial
\begin{equation}\label{eq:partial orientations}
A_G(x):=\sum_{\pi\in \mathcal{P}(V)}\prod_{S\in \pi} \text{ao}(G[S],\min(S))x^{|S|-1},    
\end{equation}
where for $S\in \pi$, $\min(S)$ denotes the smallest element of $S$ with respect to the ordering of $V$.
A priori this polynomial may seem to depend on the choice of ordering.
Surprisingly this is not the case, as it follows directly from the next lemma.
\begin{lemma}[Greene and Zaslavsky~\cite{greene1983interpretation}]\label{lem:acyclic to chromatic}
Let $G$ be a graph with $n$ vertices. Then 
\[\chi_G(x)=x^nA_G(-1/x).\]
\end{lemma}
For completeness we will provide a proof of the lemma at the end of this section.
\begin{remark}\label{rem:AO do not depend on ordering}
Note that Lemma~\ref{lem:acyclic to chromatic} implies that for an $n$-vertex graph $G=(V,E)$ the coefficient of $x^{n-1}$ in $A_G(x)$ does not depend on the ordering.
Since this coefficient is equal to $\text{ao}(G,\min(V))$, we have $\text{ao}(G,v)=\text{ao}(G,u)$ for any two vertices $u,v\in V$.
\end{remark}

We will now use the acyclic orientations to interpret the chromatic polynomial as a generating function of certain forests, different from the BCF forests used above.

For a connected rooted graph $(H,v)$ we call a spanning tree $T$ of $(H,v)$ \emph{stable} if for each $k$ the collection of vertices at distance exactly $k$ from $v$ (in $T$) forms an independent set in $H$.
Given an orientation $\omega\in \text{AO}(H,v)$ we can construct a stable tree $T$ from it  as follows.
Fix an arbitrary ordering of the edges of $H$.
We let $V_0=\{v\}$ and once $V_0,\ldots, V_k$ have been determined, we let $V_{k+1}$ to be the collection of sources of the directed graph obtained from $H$ and its orientation by removing the vertices in $V_0\cup \ldots\cup V_k$. (Where we note that isolated vertices are sources by definition.)
The tree $T$ is obtained by adding edges between $V_k$ and $V_{k+1}$ for $k=0,\ldots$ as follows: by construction, each element in $V_{k+1}$ has at least one incoming edge from $V_k$ and we select the smallest such edge.
Clearly, the resulting tree $T$ is stable.
We denote the resulting tree $T=\phi(\omega)$, suppressing the dependence of the ordering of the edges of $H$.

\begin{lemma}\label{lem:AO --> stable}
Let $(H,v)$ be a connected rooted graph and fix an ordering of the edges of $H$.
The map $\phi:\text{AO}(H,v)\to \{\text{stable spanning trees of } (H,v)\}$ is injective.
\end{lemma}
\begin{proof}
Let $T=\phi(\omega)$ for some orientation $\omega\in\text{AO}(H,v)$. 
We shall show that we can reconstruct $\omega$ from $T$ as follows.
Let for $k\in \mathbb{Z}_{\geq 0}$, $V_k$ be the set of vertices in $G$ that are at distance exactly $k$ from $v$ in $T$.
Orient all edges in $T$ from $V_k$ to $V_{k+1}$ for $k=0,1,\ldots$.
Any other edge of $H$ is of the form $\{u,u'\}$ with $u\in V_k$ and $u'\in V_{\ell}$ for some $k\neq \ell$ since $T$ is stable. We orient this edge towards $u'$ if $\ell>k$ and towards $u$ otherwise. 
The resulting orientation $\omega'$ is acyclic by construction and has $v$ as its unique source.
It follows inductively that $V_{k}$ is the collection of sources of the directed graph obtained from $H$ and $\omega'$ by removing the vertices in $V_0\cup \ldots\cup V_{k-1}$.
This implies that $\omega=\omega'$ and finishes the proof.
\end{proof}

Let now $G=(V,E)$ be a graph with a fixed ordering of its edges and of its vertices.
We denote by $\mathcal{F}^\star_G$ the collection of (edge sets of) forests $F$ of $G$ such that for each component $T=(S,T)$ of $F$, $T$ is the image of $\phi$ of an acyclic orientation in $\text{AO}(G[S],\min(S))$.
We will refer to these forests as \emph{$\star$-forests}.
The next lemma gives us our desired interpretation of the coefficients of $\chi_G$ in terms of the number of $\star$-forests.  
\begin{lemma}\label{lem:star forests}
Let $G$ be an $n$-vertex graph with a given ordering of its vertices and edges. Then
\begin{equation}\label{eq:star=bcf}
\sum_{F\in \mathcal{F}^\star_G} x^{|F|}=F_G(x)=(-x)^n\chi_G(-1/x).
\end{equation}
\end{lemma}
\begin{proof}
By construction we have
\begin{equation}
A_G(x)=\sum_{F\in \mathcal{F}^\star_G} x^{|F|}.    
\end{equation}
By Lemma~\ref{lem:acyclic to chromatic} and Theorem~\ref{thm:whitney} it thus follows that~\eqref{eq:star=bcf} holds. 
\end{proof}

\begin{remark}\label{rem:star trees do not depend on ordering}
We note that by Remark~\ref{rem:AO do not depend on ordering} for an $n$-vertex graph $G=(V,E)$ the coefficient of $x^{n-1}$ in $\sum_{F\in \mathcal{F}^\star_G} x^{|F|}$ is equal to the number of $\star$-trees of $G$ for any ordering of $G$.        
\end{remark}

\begin{remark}
As was pointed out by one of the anonymous referees, it might be possible to interpret $\mathcal{F}^\star_G$ as the collection of forests arising from the \emph{Penrose partition scheme}~\cite{ProcaccietalremarkBCF}. 
Such an interpretation was used in~\cite{fialho2026zero} to provide improved bounds on the chromatic zeros for a subclass of claw-free graphs, where additionally two other small graphs are forbidden as induced subgraphs, namely the square and the diamond.     
\end{remark}

\subsection{Proof of Lemma~\ref{lem:acyclic to chromatic}}
We conclude this section by giving a proof of Lemma~\ref{lem:acyclic to chromatic}.

\begin{proof}[Proof of Lemma~\ref{lem:acyclic to chromatic}]
We prove this by induction on the number of edges of $G$. Let $n$ be the number of vertices of $G$.
In case the number of edges equals zero, the statement clearly holds.
Now consider a graph $G=(V,E)$ with a positive number of edges and an ordering of its vertices.
We may assume that $G$ is connected since both $A_G$ and $\chi_G$ are multiplicative over the components of $G$.
Let $e=\{u,v\}$ be an edge of $G$ with $u=\min(V)$.
Then we claim that 
\begin{equation}
A_{G}(x)=xA_{G/e}(x)+A_{G\setminus e}(x), \label{eq:del con}
\end{equation}
where we take the natural ordering on the vertices of $G/e$ coming from $V$ by taking the vertex obtained form contracting $e$ as the smallest vertex and, abusing notation, call this vertex $u$. We will prove~\eqref{eq:del con} below. 
First we note that it implies that
\begin{align*}
q^nA_G(-1/q)=q^nA_{G\setminus e}(-1/q)-q^{n-1}A_{G/e}(-1/q)=\chi_{G\setminus e}(q)-\chi_{G/e}(q)=\chi_G(q),
\end{align*}
where the first identity is by induction and the second by the well known deletion-contraction recurrence for the chromatic polynomial~\eqref{eq:classic del con}.

It remains to prove~\eqref{eq:del con}. Note that by definition we have
\[
    A_G(x)=\sum_{S: u\in S} \text{ao}(G[S],u) x^{|S|-1} A_{G\setminus S}(x),
\]
where the summation is over subsets $S$ of $V(G)$ containing $u$ such that $G[S]$ is connected. Now let us further break this summation into the following three parts:
\begin{align*}
V_0&=\{S\subseteq V(G)~|~u\in S, v\notin S \textrm{ and $G[S]$ is connected}\}, \\
V_1&=\{S\subseteq V(G)~|~u,v\in S \textrm{, $G[S]$ is connected and $e$ is a cut edge of $G[S]$ }\}, \\
V_2&=\{S\subseteq V(G)~|~u,v\in S \textrm{, $G[S]$ is connected and $e$ is not a cut edge of $G[S]$ }\}.
\end{align*}

If $S\in V_0$, then trivially $\text{ao}(G[S],u)=\text{ao}((G\setminus e)[S],u)$, and thus
\begin{align}\label{eq:ao_0}
    \sum_{S\in V_0} \text{ao}(G[S],u) x^{|S|-1} A_{G\setminus S}(x)=\sum_{S\in V_0} \text{ao}((G\setminus e)[S],u) x^{|S|-1} A_{(G\setminus e)\setminus S}(x).
\end{align}

If $S\in V_1$, then any $\omega\in AO(G[S],u)$ yields an acylic orientation $\omega'\in \text{AO}((G/e)[S-v],u)$ by deleting the oriented edge $e$, where we note that all edges incident to $v$ except $e$ are oriented outwards in $\omega$, as otherwise there would be another source.
Conversely, any acyclic orientation $\omega'$ in $\text{AO}((G/e)[S-v],u)$ yields an acyclic  orientation $\omega\in \text{AO}(G[S],u)$ by orienting $e$ towards $v$. Indeed, since in $\omega'$ all edges are oriented outwards from $u$, therefore  we have that $u$ is the unique source in $\omega$.
This means that
\begin{align*}
    \text{ao}(G[S],u)=\text{ao}((G/e)[S-v],u),
\end{align*}
and thus,
\begin{align}
    \sum_{S\in V_1} \text{ao}(G[S],u) x^{|S|-1} A_{G\setminus S}(x)
    &=\sum_{\substack{u\in S'\subseteq V(G/e) \\ S'\cup\{v\}\in V_1}} \text{ao}((G/e)[S'],u) x^{|S'|} A_{(G/e)\setminus S'}(x).\label{eq:ao_1}
\end{align}

If $S\in V_2$, then we know that $S$ is connected in $G\setminus e$.  
For any $\omega\in \text{AO}(G[S],u)$ the restriction of $\omega$ to $(G\setminus e)[S]$ forms an acyclic orientation with $u$ as its unique source precisely when not all other edges incident to $v$ in $(G\setminus e)[S]$ are all oriented outwards. 
In case all other edges incident to $v$ in $(G\setminus )[S]$ are all oriented outwards, then $\omega$ restricts to an acyclic orientation in $\text{AO}((G/e)[S-v],u)$. 
Conversely, any $\omega\in \text{AO}((G/e)[S-v],u)\cup \text{AO}((G\setminus e)[S],u)$ yields an acyclic orientation in $\text{AO}(G[S],u)$ by orienting the edge $e$ towards $v$.
We conclude that \[\text{ao}(G[S],u)=\text{ao}((G/e)[S-v],u)+\text{ao}((G\setminus e)[S],u),\]
and thus
\begin{align}
    \sum_{S\in V_2} \text{ao}(G[S],u) x^{|S|-1} A_{G\setminus S}(x)=&\sum_{\substack{u\in S'\subseteq V(G/e) \\ S'\cup\{v\}\in V_2}} \text{ao}((G/e)[S'],u) x^{|S'|} A_{(G/e)\setminus S'}(x)\nonumber
    \\
    &+\sum_{\substack{S\in V_2}} \text{ao}((G\setminus e)[S],u) x^{|S|-1} A_{(G\setminus e)\setminus S}(x). \label{eq:ao_2}
\end{align}

While adding up Equation~\eqref{eq:ao_0},\eqref{eq:ao_1} and \eqref{eq:ao_2} observe that $u\in S\subseteq V(G\setminus e)$ induces a connected subgraph in $G\setminus e$ if and only if $S\in V_0\cup V_2$, while $u\in S'\subseteq V(G/e)$ induces a connected subgraph of $G/e$ if and only if $S'\cup\{v\}\in V_1\cup V_2$. This means that

\begin{align*}
A_G(x)&=\sum_{i=0}^2\sum_{S\in V_i} \text{ao}(G[S],u) x^{|S|-1} A_{G\setminus S}(x)=xA_{G/e}(x)+A_{G\setminus e}(x),
\end{align*}
as we desired.
\end{proof}


\section{Preliminaries on trees}\label{sec:trees}
In this section we collect some preliminaries on tree-generating functions that will be used in the proof of our main results.

An acyclic connected graph is called a tree. For a tree $T$, we abuse notation by using $T$ to refer both to the graph and its edge set. 
For a graph $G=(V,E)$ and a subset of its vertices $U\subset V$, we denote by $\mathcal{T}_{G,U}$ the set of acyclic subgraphs $T\subset E$ of $G$ that consist of exactly $|U|$ components, each rooted at a unique vertex of $U$.
By $\mathcal{S}_{G,U}\subset \mathcal{T}_{G,U}$ we denote the set of those $T\in \mathcal{T}_{G,U}$ for which each component $C$ of $T$ is a stable tree, each rooted at a unique vertex of $U$.

For variable $x$ and $v\in V(G)$ we define the associated \emph{rooted tree-generating functions} by
\begin{align*}
T_{G,v}(x)&:=\sum_{T\in \mathcal{T}_{G,v}} x^{|T|},
\\
S_{G,v}(x)&:=\sum_{T\in \mathcal{S}_{G,v}} x^{|T|}.
\end{align*}
The following lemma shows how to bound these rooted tree-generating functions; the version for $T_{G,v}(x)$ appears somewhat implicitly in~\cite{JPSzeros} and can also be derived from the first part of Lemma 3.5 in~\cite{procacci2024zero}.
\begin{lemma}
\label{lem:tree}
Let $G=(V,E)$ be a graph of maximum degree at most $\Delta\geq 2$ and let $v\in V$ be a vertex of degree $d$.
Fix any $b\geq 0$.
If $d\leq \Delta-1$, then
\begin{align*}
T_{G,v}\left(\frac{b\left(1+\tfrac{b}{\Delta-1}\right)^{-(\Delta-1)}}{\Delta-1}\right)\leq \left(1+\tfrac{b}{\Delta-1}\right)^{d}.
\end{align*}
If $G$ is claw-free, and $d\leq \Delta,$ then
\begin{align*}
S_{G,v}\left(\frac{b}{(1+b+b^2/4)\Delta}\right)\leq 1+b+b^2/4.
\end{align*}
\end{lemma}
\begin{proof}
A short proof for a slight variation of the first statement can be found in~\cite{PRsurvey}.
We follow that proof to prove the statement provided here.

Let us write $f(b)=\left(1+\tfrac{b}{\Delta-1}\right)$.
We prove the statement by induction on the number of vertices, the base case being trivial.
Now consider a non-empty graph $G$ of maximum degree at most $\Delta$ with a vertex $v$ of degree $d\leq \Delta-1$.
Then for $c=\tfrac{bf(b)^{-\Delta-1}}{\Delta-1}$, we have
\begin{align}\label{eq:tree step one}
    T_{G,v}(c)=\sum_{S\subseteq N(v)}c^{|S|}\sum_{F\in \mathcal{T}_{G-v,S}} c^{|F|}.
\end{align}
Since each forest $F\in \mathcal{T}_{G-v,S}$ can be decomposed into $|S|$ trees in $G-v$, we can bound
\begin{align*}
\sum_{F\in \mathcal{T}_{G-v,S}} c^{|F|}\leq \prod_{s\in S}T_{G-v,s}(c),    
\end{align*}
which by induction we can further bound by $f(b)^{|S|(\Delta-1)}$, since each vertex $s\in S$ has degree at most $\Delta-1$ in the graph $G-v$.
Plugging this in into~\eqref{eq:tree step one} we find that
\[
 T_{G,v}(c)\leq \sum_{S\subseteq N(v)}c^{S}f(b)^{|S|(\Delta-1)}=\sum_{S\subseteq N(v)} \left(\frac{b}{\Delta-1}\right)^{|S|}=\left(1+\tfrac{b}{\Delta-1}\right)^{d},
\]
as desired.

To prove the second statement we will follow the same proof idea.
The proof is by induction on the number of vertices of $G$, the base case being trivial.
Let us write $f(b)=1+b+b^2/4$ and $c=\tfrac{b}{f(b)\Delta}$.
Then 
\begin{align*}
    S_{G,v}(c)&=\sum_{\substack{U\subseteq N(v)\\U\text{ independent}}}c^{|U|}\sum_{\substack{F\in \mathcal{S}_{G-v,U}\\ F\cup vU\in \mathcal{S}_{G,v}}} c^{|F|}
    \leq \sum_{\substack{U\subseteq N(v)\\U\text{ independent}}}c^{|U|}\sum_{\substack{F\in \mathcal{S}_{G-v,U}}} c^{|F|}
    \\
    &\leq \sum_{\substack{U\subseteq N(v)\\U\text{ independent}}}c^{|U|}\prod_{u\in U}S_{G-v,u}(c)\leq \sum_{\substack{U\subseteq N(v)\\U\text{ independent}}}c^{|U|}f(b)^{|U|}.
\end{align*}
Here the last two inequalities follow since each $F\in \mathcal{S}_{G-v,U}$ can be decomposed as the disjoint union of $|U|$ rooted trees in $G-v$ each of which is contained in $\mathcal{S}_{G-v,u}$ for some $u\in U$ and since $G-v$ is again claw-free it follows by induction that $S_{G-v,u}(c)\leq f(b)$ for each $u\in U$.

Since $G$ is claw-free the neighbourhood of $N(v)$ only has independent sets of size $0,1,2$. 
By Tur\'an's theorem (or rather Mantel's theorem) applied to the complement of $G[N(v)]$ (which is triangle free) we know that the number of independent sets of size $2$ in $G[N(v)]$ is at most $|N(v)|^2/4\leq \Delta^2/4$.
Therefore 
\begin{align*}
 S_{G,v}(c)\leq 1+c\Delta f(b)+c^2(\Delta^2/4) f(b)^2\leq 1+b+b^2/4,
\end{align*}
as desired.
\end{proof}

The next lemma will be used to prove our result on large girth graphs. As in \cite{JPR24} and \cite{procacci2024zero}, the girth condition will influence the type/size of the trees appearing in the rooted tree-generating function. 

Let us define for a graph $G$, $v\in V(G)$ and $k\in \mathbb{N}$ the set 
\[
    \mathcal{T}_{G,v;k}=\{T\in \mathcal{T}_{G,v} ~|\text{ there exists } u\in V(T): \textrm{dist}_{T}(u,v)\ge k\}.
\]
Let us denote the corresponding generating function as
\[
    T_{G,v;k}=\sum_{T\in\mathcal{T}_{G,v;k}}x^{|T|}.
\]
The next lemma bounds this generating function; it also be derived from the second part of Lemma 3.5 in~\cite{procacci2024zero}.
\begin{lemma}\label{lem:large_rooted_tree}
Let $\Delta\ge 2$. Then for any graph $G$ with maximum degree $\Delta$, $v\in V(G)$ of degree at most $\Delta-1$ and $b\ge 0$ we have 
\[
T_{G,v;k}\left(\frac{b\left(1+\tfrac{b}{\Delta-1}\right)^{-(\Delta-1)}}{\Delta-1}\right)\leq e^b \left(\frac{b}{1+\tfrac{b}{\Delta-1}}\right)^k.
\]
\end{lemma}
\begin{proof}
Similarly to the previous proof, let $f(b)=1+\tfrac{b}{\Delta-1}$. 
Let $\mathcal{P}_{G,v;k}$ denote the collection of paths in $G$ that are rooted at $v$ and consist of $k$ edges. Trivially, $|\mathcal{P}_{G,v;k}|\le (\Delta-1)^{k}$. Then, using Lemma~\ref{lem:tree} and writing $x=\tfrac{bf(b)^{-(\Delta-1)}}{\Delta-1}$ we have
\begin{align*}
T_{G,v; k}(x)&\le \sum_{P\in \mathcal{P}_{G,v;k}}\sum_{P\subseteq T\in\mathcal{T}_{G,v}}x^{|T|}\\
&\le \sum_{P\in \mathcal{P}_{G,v;k}}x^{|P|}\prod_{u\in V(P)}\sum_{T\in\mathcal{T}_{G-P,u}}x^{|T|}\\
&\le \sum_{P\in \mathcal{P}_{G,v;k}}x^{|P|}\prod_{u\in V(P)}f(b)^{\deg_{G-P}(u)}.
\end{align*}
Now using that for $P\in \mathcal{P}_{G,v;k}$ each vertex $u\in V(P)$ has degree at most $\Delta-2$ except the last vertex of the path, we can bound the previous quantity by
\begin{align*}
\le\left(\tfrac{b}{f(b)}\right)^kf(b)^{\Delta-1}\leq e^b \left(\tfrac{b}{f(b)}\right)^k,
\end{align*}
as desired.

\end{proof}

\section{Bounds for all graphs}\label{sec:general}
In this section, we will state and prove our main result for all graphs from which we will derive the bounds stated in Theorem~\ref{thm:main} and Theorem~\ref{thm:main large girth} in the introduction.
We will only use the interpretation of $F_G(x)$ as the generating function of BCF forests in a graph $G$.
For a graph $G$ with a given ordering of its edges and a  vertex $v$ we denote by $\mathcal{F}_{G,v}$ the collection of BCF trees that are rooted at $v$.

In our proof we will need the following lemma.

\begin{lemma}\label{lem:expand G/e}
Let $G=(V,E)$ be a graph with an ordering of the edges and let $e=\{u,v\}$ be the smallest edge. Assume that the other edges incident to $v$ are the largest edges.   
Then
\begin{equation}\label{eq:contraction recurrence}
F_{G/e}(x)=\sum_{\substack{T\in \mathcal{F}_{G\setminus e,u}\\ V(T)\cap N_{G\setminus e}(v)=\emptyset}}x^{|T|}F_{G\setminus V(T)}(x),    
\end{equation}
where $N_{G\setminus e}(v)$ denotes the collection of neighbors of $v$ in the graph $G\setminus e$.
\end{lemma}
\begin{proof}
First, we note that since $e$ is the smallest edge in $G$, any BCF forest $F$ in $G$ containing the edge $e$ uniquely corresponds to a BCF forest $\hat F=F\setminus e$ in $G/e$.
In particular, 
\[F_{G/e}(x)=\sum_{\substack{F\in \mathcal{F}(G)\\  e\in F}} x^{|F|-1}.\]
Similarly, any BCF forest not containing $e$ uniquely corresponds to a BCF forest in $G\setminus e$.

Any BCF forest $F$ that contains $e$ can be split into a BCF tree $T$ that contains $e$ and a BCF forest $F'$ in $G\setminus V(T)$.
Let $T_u$ (resp. $T_v$) be the component of $T\setminus e$ that contains $u$ (resp. $v$). 
Then $T_u$ is a BCF tree in $G\setminus e$ and it does not intersect $N_{G\setminus e}(v)$, since otherwise $T$ would not be BCF in $G$ as some edge $f$ incident to $v$ would be the largest edge in the cycle in $T\cup f$.
Thus, we  see that the contribution of $F$ to $F_{G/e}(x)$ is the same as that of $T_u\cup (T_v\cup F')$ to the right-hand side of~\eqref{eq:contraction recurrence}.

Conversely, suppose $T_u\in \mathcal{F}_{G\setminus e,u}$  and $F'\in \mathcal{F}_{G\setminus V(T_u)}$ such that $V(T_u)$ does not intersect $N_{G\setminus e}(v)$.
Let $T_v$ be the component of $F$ that contains $v$ (note that $T_v$ could possible be empty).
Since $V(T_u)$ does not intersect $N_{G\setminus e}(v)$, we know that $T_u\cup e\cup T_v$ is BCF in $G$, because for any edge $f$ such that $T\cup f$ contains a cycle $C$, we know that $f$ is not the largest edge in $C$.
Indeed, if $C\subseteq T_u\cup f$ or if $C\subseteq T_v \cup f$, we know this to be true, since otherwise $T_u$ or $T_v$ would not be BCF. On the other hand, if $C$ uses the edge $e$, then it has to use another edge incident to $v$, say $e'$, which cannot be the edge $f$ (since $V(T_u)$ does not intersect $N_{G\setminus e}(v)$). This edge $e'$ is larger than $f$ by assumption, therefore $T$ is indeed BCF.
It follows that the contribution of $T_u\cup F'$ to the right-hand side of~\eqref{eq:contraction recurrence} is the same as that of $F=T_u \cup e\cup F'$ to the left-hand side.

This finishes the proof.
\end{proof}


Recall that for a graph $G=(V,E)$, $v\in V$ and $k\geq 0$, we denote by $T_{G,v;k}(x)$ the generating function of trees rooted at $v$ that contain a path rooted at $v$ of length at least $k$.
Define for positive integers $\Delta\geq 2$ and $g\geq 3$, $\mathcal{G}_{\Delta,g}$ to be the collection of rooted graphs $(G,v)$ of maximum degree at most $\Delta$ and girth at least $g$ such that the root, $v$, has degree at most $\Delta-1$.
Let for $b>0$, 
\begin{equation}\label{eq:define tree upper bound}
    t_{\Delta,g,i}(b):=\sup_{(G,v)\in \mathcal{G}_{\Delta,g}, \deg(v)\leq i} T_{G,v;g-2}\left(\frac{b(1+\tfrac{b}{\Delta-1})^{-(\Delta-1)}}{\Delta-1}\right),
\end{equation}
noting that by Lemma~\ref{lem:large_rooted_tree} this supremum is finite. Also note that by definition, $t_{\Delta,g,0}(b)=1$ for any $g,\Delta$ and $b$.
We now state and prove a rather technical condition guaranteeing a zero-free disk for $F_G$, after which we deduce more concrete bounds.

\begin{theorem}\label{thm:main general}
Let $\Delta\geq 2$ and $g\geq 3$ be integers. 
Let $a\in (0,1)$ and $b>0$. 
Let for $i=0,\ldots,\Delta-1,$
\begin{equation}\label{eq:define a_i}
    a_i:=\frac{b\left(1-a+ t_{\Delta,g,i}(b)\right)}{(\Delta-1) \left(1+\tfrac{b}{\Delta-1}\right)^{\Delta-1}}.
\end{equation}
Suppose that
\begin{itemize}
    \item $a_i<1$ for $i=0,\ldots,\Delta-1$, and
    \item $\prod_{i=0}^{\Delta-2}(1-a_i)\geq 1-a$.
\end{itemize}
Then for any $z\in \mathbb{C}$ such that 
\[|z|\leq \frac{b(1-a)}{(\Delta-1)(1+\tfrac{b}{\Delta-1})^{\Delta-1}}\]
and any graph $G$ of maximum degree at most $\Delta$ and girth at least $g$ we have $F_G(z)\neq 0$.
\end{theorem}
\begin{proof}
Fix $z$ such that $|z|\leq \frac{b(1-a)}{(\Delta-1)(1+\tfrac{b}{\Delta-1})^{\Delta-1}}$.
We will prove the the following three statements by induction on the number of edges:
\begin{itemize}
    \item[(i)] for any edge $e$ and any vertex $u\in e$, $\left|\frac{F_G(z)}{F_{G\setminus e}(z)}-1\right|\leq a_{\deg(u)-1}$,
    \item[(ii)] $F_G(z)\neq 0$,
    \item[(iii)] for any vertex $u$ of degree at most $\Delta-1$, $\left|\frac{F_{G-u}(z)}{F_{G}(z)}\right|\leq \frac{1}{1-a}$.
\end{itemize}
Clearly, if $G$ consists of a single edge, all three statements are true, since then $\tfrac{F_G(z)}{F_{G\setminus e}}=1+z$ and so (i) follows by assumption since $a_0=\frac{b(1-a)}{(\Delta-1)(1+\tfrac{b}{\Delta-1})^{\Delta-1}}$.
Part (ii) follows since $a_0<1$ and (iii) follows since 
\[
\left |\frac{F_{G-u}(z)}{F_G(z)}\right |=\left |\frac{1}{1+z}\right |\leq \frac{1}{1-a_0}\leq \frac{1}{1-a},
\]
by assumption.

Now assume that $G$ is a graph of maximum degree at most $\Delta$ and girth at least $g$ with more than one edge.
Note that item (i) implies item (ii), since by induction $F_{G\setminus e}(z)\neq 0$ for any edge $e$ and all $a_i$ are contained in $(0,1)$.

Secondly, statements (i) and (ii) imply statement (iii).
Indeed, by (ii) $F_G(z)\neq 0$. Write $e_1,\ldots,e_d$ for the edges incident with $u$ and denote $G_d=G$ and $G_i=G_{i+1}\setminus e_{i+1}$ for $i=0,\ldots,d-1$.
Note that $G_0=G-u$ and that the degree of $u$ in $G_{i}$ is equal to $i.$
Then 
\begin{equation}\label{eq:telescope from vertex to edge}
 \frac{F_{G-u}(z)}{F_{G}(z)}=\prod_{i=1}^{d} \frac{F_{G_{i}\setminus e_i(z)}}{{F_{G_{i}}}(z)}. \end{equation}
By (i) and induction we have 
\[
\left|\frac{F_{G_i\setminus e_i}(z)}{F_{G_i}(z)}\right|\leq \frac{1}{1-a_{i-1}}.
\]
Therefore, \eqref{eq:telescope from vertex to edge} is bounded by $\prod_{i=0}^{d-1}\tfrac{1}{1-a_i}\leq \prod_{i=0}^{\Delta-2}\tfrac{1}{1-a_i}\leq \tfrac{1}{1-a}$ by the assumptions of the theorem.

We thus focus on proving item (i) for any given edge $e$ of $G$ and $u\in e$.
Let us write $e=\{u,v\}$ and choose an ordering of the edges of $G$ in which $e$ is the smallest edge and the other edges incident to $v$ are the largest edges in the ordering of the edges of $G$.
Then by the deletion-contraction recurrence~\eqref{eq:deletion/contraction F}, 
\begin{equation}\label{eq:del/con in ratio}
\frac{F_G(z)}{F_{G\setminus e}(z)}-1=z\frac{F_{G/e}(z)}{F_{G\setminus e}(z)}=z+z\frac{F_{G/e}(z)-F_{G\setminus e}(z)}{F_{G\setminus e}(z)}.
 \end{equation}
By Lemma~\ref{lem:expand G/e} the difference $F_{G\setminus e}(z)-F_{G/e}(z)$ is given by
\[
\sum_{\substack{T\in \mathcal{F}_{G\setminus e,u}\\ V(T)\cap 
N_{G\setminus e}(v)\neq \emptyset}} z^{|T|}F_{G\setminus V(T)}(z).
\]
Any tree appearing in this sum contains a path rooted at $u$ with at least $g-2\ge 1$ edges, since the girth of $G$ is at least $g$.
Therefore, using the triangle inequality, we can bound the absolute value of the left-hand side of~\eqref{eq:del/con in ratio} by
\begin{equation}\label{eq:tree gen in proof}
|z|+\sum_{\substack{T\in \mathcal{T}_{G\setminus e,u;g-2}}}|z|^{|T|+1}\left|\frac{F_{G\setminus V(T)}}{F_{G\setminus e}(z)}\right|.
\end{equation}
Fix a tree $T$ appearing in this sum and write $V(T)=\{u=u_1,\ldots,u_k\}$ such that $u_i$ has at least one neighbour in $\{u_1,\ldots,u_{i-1}\}$ for $i\geq 2$. Denote $G_0=G\setminus e$ and for $i=1,\ldots, k$ $G_i=G_{i-1}-u_i$.
Then 
\begin{equation}\label{eq:telescope}
\frac{F_{G\setminus V(T)}(z)}{F_{G\setminus e}(z)}=\prod_{i=1}^k \frac{F_{G_{i-1}-u_i}(z)}{F_{G_{i-1}}(z)}.
\end{equation}
Since the degree of $u_i$ in $G_{i-1}$ is at most $\Delta-1$, we can apply induction and (iii) to bound 
\[
\left|
\frac{F_{G\setminus V(T)}(z)}{F_{G\setminus e}(z)}\right|\leq (1-a)^{-|V(T)|}.
\]
We can thus bound~\eqref{eq:tree gen in proof} by
\begin{align*}
&|z|+\frac{|z|}{1-a}\sum_{T\in \mathcal{T}_{G\setminus e,u;g-2}}\left(\frac{|z|}{1-a}\right)^{|T|} 
\\
\leq &|z|+\frac{|z|}{1-a}T_{G\setminus e,u;g-2}\left(\frac{b}{(\Delta-1)(1+\tfrac{b}{\Delta-1})^{\Delta-1}}\right)^{|T|}
\\
\leq &|z|+\frac{|z|}{1-a}t_{\Delta,g,\deg(u)-1}(b)= a_{\deg(u)-1},
\end{align*}
as desired.
This finishes the proof.
\end{proof}

\subsection{Bounds for all graphs}
We now derive a few consequences of Theorem~\ref{thm:main general} allowing us to prove concrete bounds on the zeros of the chromatic polynomial for general graphs of bounded degree.

\begin{proposition}\label{prop:bound general}
Let $\Delta\geq 2$ be an integer. 
Suppose that $a\in (0,1)$ and $b\in [0,\Delta-1]$ are such that
\[
    \prod_{i=0}^{\Delta-2} \left(1-\frac{b\left((1+\tfrac{b}{\Delta-1})^i-a\right)}{(\Delta-1) \left(1+\tfrac{b}{\Delta-1}\right)^{\Delta-1}}\right) \geq 1-a.
\]
Then for any $z\in \mathbb{C}$ such that 
\[|z|\leq \frac{b(1-a)}{(\Delta-1)(1+\tfrac{b}{\Delta-1})^{\Delta-1}}\]
and for any graph $G$ of maximum degree at most $\Delta$, we have $F_G(z)\neq 0$.
\end{proposition}
\begin{proof}
Since any graph has girth at least $3$, we can invoke Theorem~\ref{thm:main general} with $g=3$.
It follows from Lemma~\ref{lem:tree} that $t_{\Delta,3,i}(b)\leq (1+\tfrac{b}{\Delta-1})^i-1$ for any $i=0,\ldots, \Delta-1$. Indeed, as the trees appearing in the generating function $T_{G,v;1}$ in~\eqref{eq:define tree upper bound} have at least one edge, we can subtract the constant term, which is $1$, from the value of the tree-generating function.
This implies that for each $i=0,\ldots,\Delta-1$, 
\[
a_i\leq \frac{b\left((1+\tfrac{b}{\Delta-1})^i-a\right)}{(\Delta-1) \left(1+\tfrac{b}{\Delta-1}\right)^{\Delta-1}}=:\hat{a}_i,
\]
where we recall that the $a_i$ are defined in the statement of Theorem~\ref{thm:main general}.
By Theorem~\ref{thm:main general} it thus suffices to show that $\hat{a}_i<1$, as by assumption we have $\prod_{i=0}^{\Delta-2}(1-a_i)\geq \prod_{i=0}^{\Delta-2} (1-\hat{a}_i)\geq 1-a$.
We have for each $i$, 
\[\hat{a}_i\leq \hat{a}_{\Delta-1}<\frac{b}{\Delta-1}\leq 1.
\]
This finishes the proof.
\end{proof}

Using a computer, this proposition can be utilized to give bounds on the roots of the chromatic polynomial for graphs of small maximum degree, see Table~\ref{tab:main general}.
To prove a uniform bound, valid for all degrees, we require the following technical lemma.

\begin{lemma}\label{lem:bound log(1+x)}
Let $x\in [0,1/2]$. Then $1-x> \exp(-x-x^2).$
\end{lemma}
\begin{proof}
We have
\begin{align*}
\log(1-x)&=-\left(x+\frac{x^2}{2}+\frac{x^3}{3}+\frac{x^4}{4}+\cdots\right)>-\left(x+\frac{x^2}{2}+\frac{1}{3}(x^3+x^4+\ldots)\right )
\\
&=-\left(x+\frac{x^2}{2}+\frac{x^3}{3(1+x)}\right )>-(x+x^2),
\end{align*}
using that $\tfrac{x}{3(1+x)}< 1/2$ since $x\leq 1/2.$
Taking $\exp$ on both sides we arrive at the statement of the lemma.
\end{proof}

\begin{proposition}\label{prop:main general}
Let $\Delta\geq 2$. 
Suppose $a,b\in (0,1)$ are such that $a<1-\tfrac{1}{\Delta}$ and
\[
e^{1-(ab+1)e^{-b}}\leq \tfrac{1}{1-a}.
\]
Then, if $z\in \mathbb{C}$ is such that $|z|\leq \frac{b(1-a)}{e^b\Delta}$, then $F_G(z) \neq 0$ for any graph $G$ of maximum degree at most $\Delta$.    
\end{proposition}

\begin{proof}
Let us define $\hat{a}=a+\tfrac{1-a}{\Delta}$. 
We apply Proposition~\ref{prop:bound general} with $\hat{a}$ and $b$.
Let us denote 
for $i=0,\ldots,\Delta-2$, 
\[
a_i:=\frac{b}{(\Delta-1) \left(1+\tfrac{b}{\Delta-1}\right)^{\Delta-1}}\left((1+\tfrac{b}{\Delta-1})^{i}-\hat{a}\right).
\]
It is our goal to show that 
\begin{equation}\label{eq:goal}
\prod_{i=0}^{\Delta-2}(1-a_i)\geq 1-\hat{a}.
\end{equation}
Note that since $(1-a)\tfrac{\Delta-1}{\Delta}=1-\hat{a}$ and $\left(1+\tfrac{b}{\Delta-1}\right)^{\Delta-1}\leq e^b$, our condition on $z$ implies that 
\begin{align*}
|z|<\frac{b(1-a_\Delta)}{(1+\tfrac{b}{\Delta-1})^{\Delta-1}(\Delta-1)}
\end{align*}
and therefore combined with Proposition~\ref{prop:bound general},~\eqref{eq:goal} implies our claim.

To prove~\eqref{eq:goal} note that $a_i\leq a_{\Delta-2}\leq \tfrac{b}{\Delta-1+b}\leq \tfrac{1}{\Delta}$ for each $i$ and therefore $\sum_{i=0}^{\Delta-2}a_i^2\leq  \tfrac{1}{\Delta}\sum_{i=0}^{\Delta-2}a_i$.
Moreover,
\begin{align*}
    \sum_{i=0}^{\Delta-2}a_i&=\frac{b}{(\Delta-1)\left(1+\tfrac{b}{\Delta-1}\right)^{\Delta-1}} \left(\frac{\left(1+\tfrac{b}{\Delta-1}\right)^{\Delta-1}-1}{\tfrac{b}{\Delta-1}}-(\Delta-1)\hat{a} \right)
    \\
    &=1-\frac{\hat{a} b+1}{\left(1+\tfrac{b}{\Delta-1}\right)^{\Delta-1}}.
    \end{align*}
Now applying Lemma~\ref{lem:bound log(1+x)} we have
\begin{align*}
    \prod_{i=0}^{\Delta-2}(1-a_i)&\geq \exp\left(-\sum_{i=0}^{\Delta-2}a_i -\sum_{i=0}^{\Delta-2}a_i^2\right)\geq \exp\left(-(1+\tfrac{1}{\Delta})\sum_{i=0}^{\Delta-2}a_i\right)
    \\
    &> \exp\left (-1+\frac{\hat{a} b+1}{\left(1+\tfrac{b}{\Delta-1}\right)^{\Delta-1}}-\frac{1}{\Delta}\right)
    \\
    &>\exp\left (-1+\frac{ab+1}{\left(1+\tfrac{b}{\Delta-1}\right)^{\Delta-1}}-\frac{1}{\Delta}\right)
    \\
    &>\exp\left (-1+\frac{ab+1}{e^b}\right)\left(1-\frac{1}{\Delta}\right)\geq (1-a)\frac{\Delta-1}{\Delta}=1-\hat{a},
\end{align*}
where the final inequality is by the assumption on $a$ and $b$. This finishes the proof.
\end{proof}

We can now provide a proof of Theorem~\ref{thm:main}.
\begin{proof}[Proof of Theorem~\ref{thm:main}]
Let $G$ be a graph. 
We may assume $\Delta=\Delta(G)\geq 2$, since the statement is clearly true for graphs of maximum degree at most $1$. 
Let us choose $a=0.333$ and $b=0.739$. Then it can be verified that 
\[
(1-a)e^{1-(ab+1)e^{-b}}\leq 1.
\]
Consequently, Proposition~\ref{prop:main general} implies that if 
\[
|z|\leq \frac{1}{4.25\Delta}\leq \frac{(1-a)b}{e^b\Delta},
\]
we have $F_G(z)\neq 0$.
By Theorem~\ref{thm:whitney} it thus follows that $\chi_G(q)\neq 0$ provided $|q|\ge 4.25\Delta$.
\end{proof}
\subsection{Bounds for large girth graphs}
We next derive some consequence of Theorem~\ref{thm:main general} for large girth graphs allowing us to prove concrete bounds on their chromatic zeros.

\begin{proposition}\label{prop:main general large girth}
    Let $\Delta\ge 3$ be an integer. Suppose that $a\in (0,1)$ and $b\in(0,\tfrac{\Delta-1}{\Delta-2}]$ are such that  
    \[
       \left(1-\frac{b(1-a)}{(\Delta-1)(1+\tfrac{b}{\Delta-1})^{\Delta-1}}\right)^{\Delta-1} \ge  1-a.
    \]
Then for any $\varepsilon\in (0,1)$ there exists $g=g(\Delta,\varepsilon)$ such that if $z\in\mathbb{C}$ satisfies \[|z|\le (1-\varepsilon)\frac{(1-a)b}{(1+\tfrac{b}{\Delta-1})^{\Delta-1}(\Delta-1)},\] then $F_G(z)\neq 0$ for any graph $G$ of maximum degree at most $\Delta$  and girth at least $g$.
\end{proposition}
\begin{proof}
First let us make an observation. The function $h(b)=b(1+\tfrac{b}{\Delta-1})^{-(\Delta-1)}$ is a strictly monotone increasing function on $[0,\tfrac{\Delta-1}{\Delta-2})$, since $h'(b)$ is positive for any $0\le b<\tfrac{\Delta-1}{\Delta-2}$. This means we can choose a $0<b'<b$ such that $h(b')=h(b)(1-\varepsilon)^{1/2}$.

Next let us define $a_\varepsilon=1-(1-\varepsilon)^{1/2}(1-a)$. 
We will now verify the condition of Theorem~\ref{thm:main general} for $a_\varepsilon$ and $b'$.

Let us write $f(b)=1+\tfrac{b}{\Delta-1}$ and observe that since $b'<b\leq \tfrac{\Delta-1}{\Delta-2}$, we have $\tfrac{b'}{f(b')}<1$.
Consequently, for $g$ large enough, we have $e^{b'}\left(\tfrac{b'}{f(b')}\right)^{g-2}\leq a_\varepsilon-a$. 
Therefore, we have by Lemma~\ref{lem:large_rooted_tree} that $t_{\Delta,g,i}(b')\leq a_\varepsilon-a$ for all $i=0,\ldots,\Delta-2$.
This implies
\begin{align*}
    \prod_{i=0}^{\Delta-2}\left(1-\frac{b'\left(1-a_\varepsilon+t_{\Delta,g,i}(b')\right)}{(\Delta-1)(1+\tfrac{b'}{\Delta-1})^{\Delta-1}}\right)&\geq \left(1-\frac{b'(1-a)}{(\Delta-1)(1+\tfrac{b'}{\Delta-1})^{\Delta-1}} \right)^{\Delta-1}
    \\
    &=\left(1-\frac{1-a}{\Delta-1}h(b')\right)^{\Delta-1}>\left(1-\frac{1-a}{\Delta-1}h(b)\right)^{\Delta-1}
    \\
    &=\left(1-\frac{b(1-a)}{(\Delta-1)(1+\tfrac{b}{\Delta-1})^{\Delta-1}}\right)^{\Delta-1}
    \\
    &\geq 1-a>1-a_\varepsilon.
\end{align*}
The result now follows from Theorem~\ref{thm:main general}, as
\[
|z|\le\frac{b'(1-a_\varepsilon)}{(1+\frac{b'}{\Delta-1})^{\Delta-1}(\Delta-1)}=(1-\varepsilon)\frac{b(1-a)}{(1+\frac{b}{\Delta-1})^{\Delta-1}(\Delta-1)}.
\]



\end{proof}

The next proposition gives an easier condition to check.

\begin{proposition}
    Let $\Delta\ge 3$ be an integer. Suppose that $a\in (0,1)$ and $b\in(0,1]$ are such that
    \[
        e^{-b(1-a)e^{-b}}\ge 1-a.
    \]
    Then for any $\varepsilon\in(0,1)$ there exists $g=g(\Delta,\varepsilon)>0$ such that if $z\in\mathbb{C}$ is such that $|z|\le (1-\varepsilon)\frac{(1-a)b}{\Delta e^b}$, then $F_G(z)\neq 0$ for any graph $G$ of maximum degree at most $\Delta$ and girth at least $g$.
\end{proposition}
\begin{proof}
Let us define $a_\Delta=1-(1-a)\frac{(\Delta-1)(1+\tfrac{b}{\Delta-1})^{\Delta-1}}{\Delta e^b}$. Observe that $1>a_\Delta>a$. Now we will verify the condition of the previous proposition.
We have
\begin{align*}
    \left(1-\tfrac{b(1-a_\Delta)}{(\Delta-1)(1+\tfrac{b}{\Delta-1})^{\Delta-1}}\right)^{\Delta-1}&= (1-\tfrac{b(1-a)}{\Delta e^b})^{\Delta-1}
    \\
    &\ge \exp\left(\left(-\tfrac{b(1-a)}{\Delta e^b}-\left(\tfrac{b(1-a)}{\Delta e^b}\right)^2\right)(\Delta-1)\right)
    \\
    &\ge \exp\left(-\tfrac{(\Delta-1)(\Delta+1)}{\Delta^2}\tfrac{b(1-a)}{e^b}\right)\\
    &\ge \exp\left(-\tfrac{b(1-a)}{e^b}\right)\\
    &= 1-a\ge 1-a_\Delta,
\end{align*}
where the second inequality follows from Lemma~\ref{lem:bound log(1+x)}.
\end{proof}

We can now proof Theorem~\ref{thm:main large girth}.
\begin{proof}[Proof of Theorem~\ref{thm:main large girth}]
Let us apply the previous proposition with $b=1$ and $a=1-e W(\tfrac{1}{e})\approx 0.243$, where $W(z)$ is the Lambert W-function. Then it can be easily verified that
\[
    e^{-(1-a)be^{-b}}=e^{-W(1/e)}=\frac{W(1/e)}{1/e}=1-a.
\]
Therefore, by the previous proposition and Theorem~\ref{thm:whitney}, for any integer $\Delta\geq 3$ if
\[
    |q|>\frac{e^b\Delta}{(1-a)b}=\frac{\Delta}{W(1/e)},
\]
there exists $g>0$ such that $\chi_G(q)\neq 0$ for all graphs $G$ of maximum degree at most $\Delta$ and girth at least $g$. 
This implies that 
$K_{\Delta,\infty}\leq \tfrac{1}{W(1/e)}$ for each $\Delta\geq 3$.
\end{proof}

\section{Claw-free graphs}\label{sec:claw-free}
In this section we will state and proof our main result for claw-free graphs from which we will derive the bound stated as Theorem~\ref{thm:main CF} in the introduction.
We will use the interpretation of $F_G(x)$ as the generating function of the $\star$-forests as in Lemma~\ref{lem:star forests}.

Let $G=(V,E)$ be graph with a fixed ordering of its edges and its vertices.
Recall that $\mathcal{F}^\star_G$ is the set of all $\star$-forests of $G$.
For a set of vertices $U \subseteq V$, we write $\mathcal{F}^\star_{G,U}$ (resp. $\mathcal{F}_{G,U}$) for those $F \in \mathcal{F}^\star_G$ (resp. $F \in \mathcal{F}_G$) such that every non-trivial component of $(V,F)$ (i.e. a component that is not an isolated vertex) contains at least one vertex of $U$. 
We note that both $\mathcal{F}^\star_{G,U}$ and $\mathcal{F}_{G,U}$ contain the empty set of edges. 
So, for example, if $U$ is just a single vertex $u$, then $\mathcal{F}^\star_{G,u}$ is the set of (edge sets of) $\star$-trees that contain $u$ along with the empty set of edges. 

The following identity is fundamental for our proof.
\begin{align}\label{eq:expand U}
F_{G}(x)=\sum_{F\in \mathcal{F}^\star_{G,U}}x^{|F|}F_{G\setminus V_U(F)}(x),
\end{align}
where by $V_U(F)$ we mean the union of $U$ together with the vertex sets of the non-trivial connected components of the graph $(V,F)$ that intersect $U$.
The identity~\eqref{eq:expand U} follows from collecting $\star$-forests according to their intersection with $U$; each such a forest splits uniquely into a forest $F$ whose non-trivial components intersect $U$ and a forest contained in $G\setminus V_U(F)$.
Specializing this identity to $U=\{v\}$, 
we obtain
\begin{align}\label{eq:fundamental recurrence}
F_{G}(x)=F_{G-v}(x)+\sum_{\substack{T\in \mathcal{F}^\star_{G,v} \\ |T|\geq 1}}x^{|T|}F_{G\setminus V(T)}(x),
\end{align}

The following inequality will be useful. 
For any positive number $y$ and $S\subseteq V$ we have
\begin{equation}
\label{eq:Forest-to-tree}
    \sum_{F \in \mathcal{F}^\star_{G,S}}y^{|F|} \leq \prod_{s \in S}S_{G,s}(y).
\end{equation}
This holds because we can group the forests $F$ in $\mathcal{F}^\star_{G,S}$ by the vertex sets of their non-trivial components  $C_1,\ldots,C_k$. 
Such a vertex set $C_i$ intersects $S$ in at least one vertex, let one of those be $s_i$.
By Remark~\ref{rem:star trees do not depend on ordering} the number of spanning $\star$-trees in $G[C_i]$ is equal to the number of spanning $\star$-trees, where $C_i$ is reordered such that $s_i$ is the smallest element.
Since each such a $\star$-tree is stable we can bound their number by the number of stable trees of size $|C_i|-1$ in $\mathcal{S}_{G,s_i}$.
It follows that \[
\sum_{F \in \mathcal{F}^\star_{G,S}}y^{|F|} \leq \prod_{s \in S} \sum_{T \in \mathcal{S}_{G,s}}y^{|T|}\leq \prod_{s \in S}S_{G,s}(y),\]
as desired.

In what follows, we denote for a graph $G$, $v\in V$ and $S\subseteq V$ by $vS$ the collection of edges of $G$ of the form $\{v,s\}$ with $s\in S$.
We start with the following useful lemma.
\begin{lemma}\label{lem:structure bad}
Let $G=(V,E)$ be a graph and let $v\in V$. 
Let $U=\{u_1,u_2\}\subseteq N(v)$ be an independent set.
Assume $G$ is equipped with an ordering of its vertices and edges in which $v$ is the smallest vertex, such that inside $G-v$, $u_1<u_2$ are the smallest vertices in the order and the edges are ordered lexicographically based on their vertices.
Let $F\in \mathcal{F}^\star_{G-v,U}$.
If $vU\cup F\notin \mathcal{F}^\star_{G,v}$, then $|F|\geq 1$ and if additionally $F$ consists of one non-trivial component, then this component contains $u_2$, but not $u_1$.
\end{lemma}
\begin{proof}
Denote $T=vU\cup F$ and assume that $T\notin \mathcal{F}^\star_{G,v}$. 
Denote for $k\geq 0$ by $V_k$ the collection of vertices of $T$ at distance exactly $k$ from $v$. 
There are three cases: (1) $T$ is not a tree, (2) $T$ is a tree but some $V_k$ is not independent in $G$, and (3) $T$ is a tree and all $V_k$ are independent sets of $G$.

In case (1), $T$ is not a tree and therefore it must contain a cycle $C$. Since $F$ is a forest the cycle must use the vertex $v$ and two vertices of $U$. Since $U$ is independent the cycle has length at least $4$ and therefore $C\setminus vU$ is a path of length at least $2$ and is contained in $F$ and hence $|F|\geq 2$. Moreover, $F$ consists of at least two non-trivial components.

In case (2), we have that $k\geq 2$ since $V_1\subseteq U$ is independent.
This means that some edge $e=\{x,y\}$ of $G$ is contained in $V_k$. Since $F\in \mathcal{F}^\star_{G-v,U}$ it follows that $F$ must have at least two components, one containing $x$ and one containing $y$.
The edge $e$ together with the two paths in $F$ from $U$ to $x$ and $y$ forms a path $P$ of length at least $3$. Therefore $|F|\geq 2$.

In case (3), we know that $T$ is a stable tree and we can in fact construct an acyclic orientation $\omega$ of $G[V(T)]$ from $T$ by orienting all edges away from the root $v$.
(Note that the orientation only depends on the sets $V_k$.)
The only obstruction for $T$ being an element of $\mathcal{F}^\star_{G,v}$ comes from the issue that the tree constructed from this orientation is not $T$, meaning that for some $k\geq 2$ there must be a $y\in V_k$ such that the unique edge between $V_{k-1}$ and $y\in V_k$ in $T$ (call it $\{x,y\} $) is not the smallest possible edge from $V_{k-1}$ to $y$. Adding this smallest edge to $T$ gives us a cycle $C$ containing $y$ of length at least $4$ since $V_{k-1}$ is independent, moreover it implies that $|V_{k-1}|\ge 2$.

It remains to show that if $F$ consists of a single non-trivial component, then $V(F)$ must contain $u_2$ but not $u_1$. Note that $F$ consisting of a single non-trivial component implies that $k=2$, since $F\in \mathcal{F}^\star_{G-v,U}$. Otherwise the problematic $e=\{x,y\}$ edge for $T$ would already be problematic for $F$. 
This implies that $u_1$ is indeed not contained in $V(F)$, otherwise we would have $u_1=x$, and $xy$ would be the smallest edge between $V_1$ and $V_2$.
\end{proof}

\begin{theorem}\label{thm:main claw-free}
Let $a\in (0,1)$ and $b>0$ be such that
 \begin{align*}
%
    h(a,b)&:=\tfrac{b(1-a)}{1+b+b^2/4}+\left(\tfrac{b(1-a)}{2+2b+b^2/2}\right)^2+\tfrac{b^2}{4}\left(1-\tfrac{1}{1+b+b^2/4}\right)<a
    \end{align*}
and let $f(a,b):=\tfrac{b(1-a)}{1+b+b^2/4}$.
Then for any $\Delta\in \mathbb{N}$ and $x\in \mathbb{C}$ such that $|x|\leq \tfrac{f(a,b)}{\Delta}$, $F_G(x)\neq 0$ for any claw-free graph $G$ of maximum degree at most $\Delta$.
\end{theorem}
\begin{proof}
Fix $\Delta\in \mathbb{N}$.
We will inductively prove the following two statements for any $x$ such that $|x|\leq \tfrac{f(a,b)}{\Delta}$ and any claw-free graph $G$ of maximum degree at most $\Delta$:
\begin{itemize}
    \item[(i)] \label{it:claw_free i} $F_{G}(x)\neq 0$, 
    \item[(ii)] \label{it:claw_free ii} for any vertex $v\in V$ we have $\left|\frac{F_{G}(x)}{F_{G-v}(x)}-1\right|\leq a$.
\end{itemize}

We note that the base case, $G$ being a one-vertex graph is trivial.
Next let $G=(V,E)$ be a claw-free graph with more than $1$ vertex and such that $\Delta(G)\leq \Delta$.
Note that it suffices to prove item (ii), since by induction $F_{G-v}(x)\neq 0$ and thus $F_G(x)=0$ would imply $\tfrac{F_{G}(x)}{F_{G-v}(x)}=-1$, violating that $a\in (0,1)$.
We thus focus on proving item (ii). To this end let us fix an ordering of the vertices of $G$ in which $v$ is the smallest and such that its neighbors are the smallest vertices in $G-v$; we order the edges lexicographically based on their vertices.

Writing $G'=G-v$ and using~\eqref{eq:fundamental recurrence} we see that we can write
\begin{align}
\frac{F_{G}(x)}{F_{G-v}(x)}-1&=\sum_{\substack{T\in \mathcal{F}^\star_{G,v}\\|T|\geq 1}}x^{|T|}\frac{F_{G'\setminus V(T)}(x)}{F_{G'}(x)}=\sum_{}\nonumber
\\
&=\sum_{\substack{\emptyset\neq U\subseteq N(v)\\ U \text{ ind.}}}x^{|U|}\sum_{\substack{F\in \mathcal{F}_{G-v,U}\\F\cup vU \in \mathcal{F}^\star_{G,v}}}x^{|F|} \frac{F_{G'\setminus V_U(F)}(x)}{F_{G'}(x)},\label{eq:U expansion}
\end{align}
where the summation is over independent sets $U$ contained in $N(v).$
Indeed, for any tree $T\in \mathcal{F}^\star_{G,v}$ the neighbors of $v$ in $T$ intersect $N(v)$ in an independent $U$ set and after removing $v$ from this intersection the resulting graph is a forest contained in $\mathcal{F}_{G',U}$. 

Since $G$ is claw-free we only need to consider independent sets $U$ of size $1$ and $2$ in~\eqref{eq:U expansion}.
Let us fix such an independent set $U$.

In case $U=\{u\}$, we reorder the vertices of $G'$ so that $u$ becomes the smallest vertex, while the relative ordering of the other vertices remains the same (the edges are ordered lexicographically) noting that this does not affect the value of $F_{G'}(x)$ by Remark~\ref{rem:star trees do not depend on ordering}.
With this ordering we claim that there is a bijection between the set $\{F\in \mathcal{F}_{G-v,u} ~|~F\cup vu \in \mathcal{F}^\star_{G,v}\}$ and $\mathcal{F}^\star_{G',u}$. 
Indeed, for each $F\in \mathcal{F}_{G-v,u}$ such that $F\cup e\in \mathcal{F}^\star_{G,v}$ we have  $F\in \mathcal{F}^\star_{G',u}$, as the acyclic orientation $\omega$ corresponding to $F\cup e$ (which has $v$ as its unique source) restricts to an acyclic orientation corresponding to $F$ with $u$ as its unique source as follows from the way the $\star$-trees are constructed in Section~\ref{ssec:acyclic}.
Conversely, any $F\in \mathcal{F}^\star_{G',u}$ satisfies that $F\cup e\in \mathcal{F}^\star_{G,v}$.
Therefore, 
\[
\sum_{\substack{F\in \mathcal{F}_{G-v,u}\\F\cup vu \in \mathcal{F}^\star_{G,v}}}x^{|F|} \frac{F_{G'\setminus V_u(F)}(x)}{F_{G'}(x)}=\sum_{F\in \mathcal{F}^\star_{G',u}}x^{|F|} \frac{F_{G'\setminus V_u(F)}(x)}{F_{G'}(x)}=\frac{F_{G'}(x)}{F_{G'}(x)}=1,
\]
where the final equality is due to~\eqref{eq:expand U}.

In case $U=\{u_1,u_2\}$, we reorder the vertices of $G'$ so that $u_1<u_2$ and these are the smallest vertices of $G'$, while the relative ordering of the other vertices remains the same (the edges are again ordered lexicographically). Then for any $F\in \mathcal{F}_{G-v,U}$ such that $F\cup vU\in \mathcal{F}^\star_{G,v}$ we have that $F\in \mathcal{F}^\star_{G',U}$. Indeed for the acyclic orientation $\omega$ corresponding to $F\cup vU$ we have that after removing $v$ the vertices in $U$ are its two only sources. 
It then follows from the construction of $F\cup vU$ from $\omega$ that $F$ is contained in $\mathcal{F}^\star_{G',U}$.    

By applying \eqref{eq:expand U} to $G'=G-v$ we obtain,
\begin{align*}
&\sum_{\substack{F\in \mathcal{F}_{G-v,U}\\F\cup vU \in \mathcal{F}^\star_{G,v}}}x^{|F|} \frac{F_{G'\setminus V_U(F)}(x)}{F_{G'}(x)}=1-\frac{F_{G'}(x)}{F_{G'}(x)}+\sum_{\substack{F\in \mathcal{F}^\star_{G-v,U}\\F\cup vU \in \mathcal{F}^\star_{G,v}}}x^{|F|} \frac{F_{G'\setminus V_U(F)}(x)}{F_{G'}(x)}
\\
=&1-\sum_{\substack{F\in \mathcal{F}^*_{G-v,U}\\F\cup vU \notin\mathcal{F}^\star_{G,v}}}x^{|F|}\frac{F_{G'\setminus V_U(F)}(x)}{F_{G'}(x)}.
\end{align*}
By Lemma~\ref{lem:structure bad} we know that the forests $F$ in the final summation above have at least one edge and in case it consists of a single component it cannot contain the smallest vertex of $U$. 
By induction (here we use that any induced subgraph of a claw-free graph is claw-free) and a telescoping argument (as in equation~\eqref{eq:telescope} in the proof of Theorem~\ref{thm:main general}) we have 
\begin{align}\label{eq:telescope_claw_free}
    \left|\frac{F_{G'\setminus V_U(F)}(x)}{F_{G'}(x)}\right|\leq \frac{1}{(1-a)^{|F|+|U|}},
\end{align}
since $F$ has at most $|U|$ components implying that $|V_U(F)|\leq |F|+|U|$.
Therefore we can bound $\left|\frac{F_{G}(x)}{F_{G-v}(x)}-1\right|$ as follows:
\begin{align*}
    \left|\frac{F_{G}(x)}{F_{G-v}(x)}-1\right|&\leq \sum_{\substack{\emptyset \neq U\subseteq N(v)\\ |U|\leq 2 \text{ ind.}}}|x|^{|U|}+\sum_{\substack{U\subseteq N(v)\\ |U|= 2, \text{ $U$ ind.}}} \left(\tfrac{|x|}{1-a}\right)^{|U|}\sum_{\substack{F\in \mathcal{F}^\star_{G',U}\\|F|\geq 1}}\left(\tfrac{|x|}{1-a}\right)^{|F|},  
 \end{align*}
where $F$ doesn't contain the smallest element of $U$ in case it consists of a single nontrivial component.
By~\eqref{eq:Forest-to-tree} we can bound this by
    \begin{align*}
    \sum_{\substack{\emptyset \neq U\subseteq N(v)\\ |U|\leq 2 \text{ ind.}}}|x|^{|U|}+\sum_{\substack{U\subseteq N(v)\\ |U|=2,  \text{ $U$ ind.}}} \left(\tfrac{|x|}{1-a}\right)^{|U|}\left(S_{G',\min(U)}\left(\tfrac{|x|}{1-a}\right)\cdot \left(S_{G',\max(U)}\left(\tfrac{|x|}{1-a}\right)-1\right)\right).  \label{eq:bound R part 1}
\end{align*}

We now use Lemma~\ref{lem:tree} to bound $S_{G',u}\left(\tfrac{|x|}{1-a}\right)$ by $q(b):=1+b+b^2/4$.
Since by Mantel's theorem $N(v)$ contains at most $\Delta^2/4$ independent sets of size $2$, we can thus bound the second part in the sum above by
\begin{align*}
\Delta^2/4 \left(\tfrac{b}{q(b)\Delta}\right)^2(q(b)^2-q(b))=\tfrac{b^2}{4}(1-1/q(b)).
\end{align*}
The first part can be bounded by 
\begin{align*}
\frac{b(1-a)}{q(b)}+\left(\frac{b(1-a)}{2q(b)}\right)^2.
\end{align*} 
By combining these two bounds, we see that $\left|\frac{F_{G}(x)}{F_{G-v}(x)}-1\right|$ is bounded by $h(a,b)$, which by assumption is bounded by $a$.
This finishes the proof.
\end{proof}

We can now prove the bound on the chromatic zeros for claw-free graphs as stated in the introduction.

\begin{proof}[Proof of Theorem~\ref{thm:main CF}]
We use Theorem~\ref{thm:main claw-free}.
Taking $b=0.865$ and $a=0.377$, it can be numerically verified that $h(a,b)<a$ and $\tfrac{1+b+b^2/4}{(1-a)b}\leq 3.81$.
Therefore for any claw-free graph $G$ and any $|x|\leq \tfrac{1}{3.81\Delta(G)}$ we have that $F_{G}(x)\neq 0$.
By Lemma~\ref{lem:acyclic to chromatic}, it follows that $\chi_G(q)\neq 0$, provided $q\geq 3.81\Delta(G)$.
\end{proof}

\section{Concluding remarks}\label{sec:remarks}
We note that even though the proof of Theorems~\ref{thm:main} and~\ref{thm:main large girth} and that of Theorem~\ref{thm:main CF} are somewhat different in nature, they both rely on the boundedness of an associated rooted tree-generating function.
The same is true for previous results on the bounds of chromatic zeros~\cites{sokal,Fernandez,JPR24}.
One of the technical motivations for studying claw-free graphs is that the rooted stable tree-generating function has a stronger bound, cf. Lemma~\ref{lem:tree}.
Note that as mentioned in~\cite{JPR24} the bounds on the ordinary tree-generating function from Lemma~\ref{lem:tree} are essentially tight as is witnessed by the infinite regular tree.
This stronger bound for the rooted stable tree-generating function allowed us to prove better bound for the chromatic zeros of claw-free graphs than that of all graphs.
Two remarks are in order here. 

First of all, it is hard to compare the bounds from Theorems~
\ref{thm:main} and~\ref{thm:main CF} since they have been obtained by different methods. One proof is based on induction on the number of edges while the other on the number of vertices. It is not clear whether the edge-based proof can be applied to claw-free graphs. First of all, claw-free graphs are not closed under removing edges.  Secondly, the interpretation of the coefficients of $F_G(x)$ in terms of $\star$-forests does not appear to behave nicely under edge-removal.
On the other hand, the vertex-based proof can be applied to general graphs (with the interpretation of $F_G(x)$ as the generation function of $\star$-forests). This approach gives better bounds than in~\cite{JPR24}, however, they are not nearly as good as the ones we obtained with the edge-based induction.

Secondly, the improvement from $4.25$ to $3.81$ may seem moderate. In our opinion, this indicates that other methods avoiding the use of the tree-generating function should be sought for the proof of zero-freeness. 
Indeed, we believe that our bounds are not optimal, cf.~\cite{dong2022foundations}*{Conjecture 11.98}.
We thus leave it as an exciting and challenging problem to further improve the bounds on the modulus of chromatic zeros in terms of the maximum degree.

\subsection*{Acknowledgements}
We thank the two anonymous referees for carefully reading an earlier version of the paper and for their constructive feedback.

\bibliographystyle{numeric}
\bibliography{chromatic}
\end{document}